\documentclass[11pt]{amsart}
\usepackage{amsmath,amsthm,amssymb,amsxtra}
\usepackage{verbatim}
\usepackage{hyperref}

\usepackage[margin=1.35in]{geometry}

\newcommand{\af}{{\mathrm{af}}}
\newcommand{\al}{\alpha}
\newcommand{\alv}{\alpha^\vee}
\newcommand{\A}{\mathbb{A}}

\newcommand{\B}{\mathbb{P}}
\newcommand{\C}{\mathbb{C}}

\newcommand{\Frac}{\mathrm{Frac}}
\newcommand{\Fun}{\mathrm{Fun}}

\newcommand{\Gr}{\mathrm{Gr}}
\newcommand{\Hom}{\mathrm{Hom}}

\def\i{{\mathbf{i}}}
\newcommand{\id}{\mathrm{id}}
\newcommand{\ip}[2]{\langle #1\,,\,#2\rangle}

\def\j{{\mathbf{j}}}
\newcommand{\la}{\lambda}
\newcommand{\La}{\Lambda}

\newcommand{\Q}{\mathbb{Q}}
\newcommand{\pt}{\mathrm{pt}}

\newcommand{\rN}{N}

\newcommand{\sh}{\hat{\sigma}}
\newcommand\Supp{\mathrm{Supp}}
\newcommand{\Sym}{\mathrm{Sym}}

\newcommand{\thv}{{\theta^\vee}}
\newcommand{\tS}{\tilde{S}}
\newcommand{\Z}{\mathbb{Z}}

\newtheorem{prop}{Proposition}
\newtheorem{lem}[prop]{Lemma}
\newtheorem{thm}[prop]{Theorem}
\newtheorem{conj}[prop]{Conjecture}
\theoremstyle{definition}
\newtheorem{rem}[prop]{Remark}
\newtheorem{ex}[prop]{Example}

\begin{document}
\title[Affine Grassmannian Equivariant Homology Pieri Rule]{Equivariant Pieri Rule for the homology of the affine Grassmannian}

\author{Thomas Lam}
\address{Department of Mathematics,
University of Michigan, 530 Church St., Ann Arbor, MI 48109 USA}
\email{tfylam@umich.edu}
 \thanks{T.L. was supported by NSF grant DMS-0901111, and by a Sloan Fellowship.}
\author{Mark Shimozono}
\address{Department of Mathematics, Virginia Tech, Blacksburg, VA 24061-0123 USA}
\email{mshimo@vt.edu}
\thanks{M.S. was supported by NSF DMS-0652641 and DMS-0652648.}


\begin{abstract} An explicit rule is given for the product of the
degree two class with an arbitrary Schubert class in the
torus-equivariant homology of the affine Grassmannian. In addition 
a Pieri rule (the Schubert expansion of the product of
a special Schubert class with an arbitrary one) is established
for the equivariant homology of the affine Grassmannians of $SL_n$ and a similar 
formula is conjectured for $Sp_{2n}$ and $SO_{2n+1}$. For $SL_n$
the formula is explicit and positive. By a theorem of Peterson these
compute certain products of Schubert classes in the torus-equivariant quantum cohomology of flag varieties. 
The $SL_n$ Pieri rule is used in our recent definition of $k$-double Schur functions and affine double Schur functions.
\end{abstract}

\maketitle

\section{Introduction}
Let $G$ be a semisimple algebraic group over $\C$ with a Borel
subgroup $B$ and maximal torus $T$. Let
$\Gr_G=G(\C((t)))/G(\C[[t]])$ be the affine Grassmannian of $G$. The
$T$-equivariant homology $H_T(\Gr_G)$ and cohomology $H^T(\Gr_G)$ are
dual Hopf algebras over $S=H^T(\pt)$ with Pontryagin and cup
products respectively. Let $W_\af^0$ be the minimal length cosets in
$W_\af/W$ where $W_\af$ and $W$ are the affine and finite Weyl
groups. Let $\{\xi_w\mid w\in W_\af^0\}$ be the Schubert basis of
$H_T(\Gr_G)$. Define the equivariant Schubert homology structure
constants $d_{uv}^w\in S$ by
\begin{align} \label{E:homstruct}
  \xi_u \xi_v = \sum_{w\in W_\af^0} d^w_{uv} \xi_w
\end{align}
where $u,v\in W_\af^0$.   One interest in the polynomials $d^w_{uv}$ is the fact that they are precisely the Schubert
structure constants for the $T$-equivariant quantum cohomology rings
$QH^T(G/B)$ \cite{LS:Q,P}.  Due to a result of Mihalcea \cite{Mih}, they have the positivity property
\begin{align}
  d^w_{uv} \in \Z_{\ge0}[\alpha_i\mid i\in I].
\end{align}
Our first main result (Theorem \ref{T:jr0}) is an ``equivariant homology Chevalley
formula", which describes $d_{r_0,v}^w$ for an arbitrary affine
Grassmannian. Our second main result (Theorem \ref{T:Pieri}) is an ``equivariant homology Pieri formula" for $G=SL_n$, which is a manifestly positive formula
for $d_{\sigma_m,v}^w$ where the homology classes
$\{\xi_{\sigma_m}\mid 1\le m\le n-1\}$ are the special classes that
generate $H_T(\Gr_{SL_n})$. In a separate work \cite{LS:Pieri} we use this Pieri formula to define new
symmetric functions, called $k$-double Schur functions and affine double Schur functions, which represent the equivariant Schubert
homology and cohomology classes for $\Gr_{SL_n}$.

\section{The equivariant homology of $\Gr_G$}
We recall Peterson's construction \cite{P} of the equivariant
Schubert basis $\{j_w\mid w\in W_\af^0\}$ of $H_T(\Gr_G)$ using the
level-zero variant of the Kostant and Kumar (graded) nilHecke ring
\cite{KK:H}. We also describe the equivariant localizations of
Schubert cohomology classes for the affine flag ind-scheme in
terms of the nilHecke ring; these are an important ingredient in our
equivariant Chevalley and Pieri rules.

\subsection{Peterson's level zero affine nilHecke ring}
Let $I$ and $I_\af = I\cup \{0\}$ be the finite and affine Dynkin
node sets and $(a_{ij}\mid i,j\in I_\af)$ the affine Cartan matrix.

Let $P_\af = \Z\delta \oplus \bigoplus_{i\in I_\af} \Z \La_i$ be the
affine weight lattice, with $\delta$ the null root and $\La_i$ the affine fundamental weight.
The dual lattice $P_\af^*=\Hom_\Z(P_\af,\Z)$ 
has dual basis $\{d\} \cup \{\alpha_i^\vee\mid i\in I_\af\}$
where $d$ is the degree generator and $\alpha_i^\vee$ is a simple coroot.
The simple roots $\{\alpha_i\mid i\in I_\af\}\subset P_\af$ are defined by
$\alpha_j = \delta_{j0} \delta + \sum_{i\in I_\af} a_{ij} \La_i$ for $j\in I_\af$ where $(a_{ij}\mid i,j\in I_\af)$
is the affine Cartan matrix. Then $a_{ij}=\ip{\alpha_i^\vee}{\alpha_j}$ for all $i,j\in I_\af$.
Let $(a_i\mid i\in I_\af)$ (resp. $(a_i^\vee\mid i\in I_\af)$) 
be the tuple of relatively prime positive integers giving a relation among the columns (resp. rows)
of the affine Cartan matrix. Then $\delta = \sum_{i\in I_\af} a_i \alpha_i$. Let $c=\sum_{i\in I_\af} a_i^\vee \alpha_i^\vee\in P_\af^*$
be the canonical central element. The level of a weight $\la\in P_\af$ is defined by $\ip{c}{\la}$.

There is a canonical projection $P_\af \to P$ where $P$ is the finite weight lattice, with kernel $\Z \delta \oplus \Z \La_0$.
There is a section $P\to P_\af$ of this projection whose image lies in the sublattice of $\bigoplus_{i\in I_\af} \Z \La_i$
consisting of level zero weights. We regard $P\subset P_\af$ via this section.

Let $W$ and $W_\af$ denote the finite and affine Weyl groups. 
Denote by $\{r_i\mid i\in I_\af \}$ the simple generators of $W_\af$. 
$W_\af$ acts on $P_\af$ by $r_i \cdot \la = \la - \ip{\alpha_i^\vee}{\la} \alpha_i$
for $i\in I_\af$ and $\la\in P_\af$. $W_\af$ acts on $P_\af^*$ by $r_i\cdot \mu = \mu - \ip{\mu}{\alpha_i} \alpha_i^\vee$
for $i\in I_\af$ and $\mu\in P_\af^*$. There is an isomorphism $W_\af\cong W\ltimes Q^\vee$
where $Q^\vee = \bigoplus_{i\in I} \Z \alpha_i^\vee \subset P_\af^*$ is the finite coroot lattice.
The embedding $Q^\vee\to W_\af$ is denoted $\mu\mapsto t_\mu$.
The set of real affine roots is $W_\af \cdot \{\alpha_i\mid i\in I_\af\}$. For a real affine root $\alpha=w\cdot \alpha_i$,
the associated coroot is well-defined by $\alpha^\vee = w\cdot \alpha_i^\vee$.

Let $S=\Sym(P)$ be the symmetric algebra, and $Q=\Frac(S)$ the fraction field. 
$W_\af \cong W \ltimes Q^\vee$ acts on $P$ (and therefore on $S$ and on
$Q$) by the level zero action:
\begin{align} \label{E:lev0action}
  w t_\mu \cdot \la = w \cdot \la\qquad\text{for $w\in W$ and $\mu\in Q^\vee$.}
\end{align}
Since $t_{-\theta^\vee} = r_\theta r_0$ we have
\begin{align} \label{E:r0theta}
  r_0 \cdot \la = r_\theta \cdot \la\qquad\text{for $\la\in P$.}
\end{align}
Finally, we have $\delta = \alpha_0 + \theta$ where $\theta\in P$ is the highest root.
So under the projection $P_\af\to P$, $\alpha_0\mapsto -\theta$.

Let $Q_{W_\af} = \bigoplus_{w\in W_\af} Q w$ be the skew group ring,
the $Q$-vector space $Q\otimes_{\mathbb{Q}} \Q[W_\af]$ with
$Q$-basis $W_\af$ and product $(p \otimes v)(q\otimes w) = p(v\cdot
q) \otimes vw$ for $p,q\in Q$ and $v,w\in W_\af$. $Q_{W_\af}$ acts
on $Q$: $q\in Q$ acts by left multiplication and $W_\af$ acts as
above.

For $i\in I_\af$ define the element $A_i\in Q_{W_\af}$ by
\begin{align} \label{E:Adef}
  A_i = \al_i^{-1}(1-r_i).
\end{align}
$A_i$ acts on $S$ since
\begin{align}
  A_i \cdot \la &= \ip{\alv_i}{\la} &\qquad&\text{for $\la \in P$} \\
  A_i \cdot (ss') &= (A_i \cdot s) s' + (r_i\cdot s)(A_i\cdot s') &\qquad&\text{for $s,s'
  \in S$.}
\end{align}
The $A_i$ satisfy $A_i^2=0$ and
\begin{align*}
  \underbrace{A_iA_jA_i\dotsm}_{\text{$m_{ij}$ times}} = \underbrace{A_jA_iA_j\dotsm}_{\text{$m_{ij}$ times}}
\end{align*}
where
\begin{align*}
  \underbrace{r_ir_jr_i\dotsm}_{\text{$m_{ij}$ times}} = \underbrace{r_jr_ir_j\dotsm}_{\text{$m_{ij}$ times}}.
\end{align*}
For $w\in W_\af$ we define $A_w$ by
\begin{align}
\label{E:Awdef}
A_w &= A_{i_1}A_{i_2}\dotsm A_{i_\ell} &\qquad&\text{where} \\
\label{E:reduced}
 w &= r_{i_1}r_{i_2}\dotsm r_{i_\ell} &\qquad&\text{is reduced.}
\end{align}
The level zero graded affine nilHecke ring $\A$ (Peterson's \cite{P}
variant of the nilHecke ring of Kostant and Kumar \cite{KK:H} for an affine root system) is
the subring of $Q_{W_\af}$ generated by $S$ and $\{A_i\mid i\in
I_\af\}$. In $\A$ we have the commutation relation
\begin{align}\label{E:nilHeckerelation}
  A_i \la = (A_i\cdot \la) 1 + (r_i\cdot\la) A_i\qquad\text{for $\la\in P$.}
\end{align}
In particular
\begin{align}
  \A = \bigoplus_{w\in W_\af} S A_w.
\end{align}

\subsection{Localizations of equivariant cohomology classes}
Using the relation
\begin{align} \label{E:rinA}
r_i = 1 - \al_i A_i
\end{align}
$w\in W_\af$ may be regarded as an element of $\A$. For $v,w\in W$
define the elements $\xi^v(w)\in S$ by
\begin{align}
  w = \sum_{v\in W} (-1)^{\ell(v)} \xi^v(w) A_v.
\end{align}
Using a reduced decomposition \eqref{E:reduced} for $w$ and
substituting \eqref{E:rinA} for its simple reflections, one obtains the formula \cite{AJS} \cite{B}
\begin{align} \label{E:Billey}
\xi^v(w) = \sum_{b\in [0,1]^\ell} \left(\prod_{j=1}^\ell
\alpha_{i_j}^{b_j} r_{i_j} \right) \cdot 1
\end{align}
where the sum runs over $b$ such that $\prod_{b_j=1} r_{i_j} = v$ is
reduced and the product over $j$ is an ordered left-to-right product of operators.
Each $b$ encodes a way to obtain a reduced word for $v$ as
an embedded subword of the given reduced word of $w$: if $b_j=1$
then the reflection $r_{i_j}$ is included in the reduced word for
$v$. Given a fixed $b$ and an index $j$ such that $b_j=1$, the root
associated to the reflection $r_{i_j}$ is by definition
$r_{i_1}r_{i_2}\dotsm r_{i_{j-1}} \cdot\alpha_{i_j}$. The summand
for $b$ is the product of the roots associated to reflections in the
given embedded subword.

It is immediate that
\begin{align}
\label{E:xisupport}
  \xi^v(w)&=0\qquad\text{unless $v\le w$} \\
\label{E:xiid}
  \xi^\id(w) &= 1 \qquad\text{for all $w$.}
\end{align}
The element $\xi^v(w)\in S$ has the following geometric
interpretation. Let $X_\af=G_\af/B_\af$ be the Kac-Moody flag ind-variety
of affine type \cite{Ku}. For every $v\in W_\af$ there is a $T$-equivariant
cohomology class $[X_v]\in H^T(X_\af)$ and for each $w\in W_\af$
there is an associated $T$-fixed point (denoted $w$) in $X_\af$
and a localization map $i_w^*:H^T(X_\af)\to H^T(\pt)$ \cite{Ku}.
Then $\xi^v(w) = i_w^*([X_v])$. Moreover the map
$H^T(X_\af) \to H^T(W_\af)\cong \Fun(W_\af,S)$ given by restriction
of a class to the $T$-fixed subset $W_\af\subset X_\af$, is an
injective $S$-algebra homomorphism where $\Fun(W_\af,S)$ is the
$S$-algebra of functions $W_\af\to S$ with pointwise product. 
The function $\xi^v\in \Fun(W_\af,S)$ is the image of $[X_v]$. The
image $\Phi$ of $H^T(X_\af)$ in $\Fun(W_\af,S)$ satisfies the GKM
condition \cite{GKM} \cite{KK:H}: For $f\in \Phi$ we
have\footnote{Using equivariance for the maximal torus $T_\af\subset
G_\af$, the GKM condition characterizes the image of localization to
torus fixed points. However after forgetting equivariance down to
the smaller torus $T$, elements of $\Phi$ are characterized by
additional conditions, which were determined in \cite{GKM2}.}
\begin{align} \label{E:GKM}
  f(w) - f(r_\beta w) \in \beta S \qquad\text{for all $w\in W_\af$ and affine real
  roots $\beta$.}
\end{align}

\begin{lem} \label{L:xiprops} Suppose $u,v\in W$ with
$\ell(uv)=\ell(u)+\ell(v)$.
\begin{align}
\label{E:productroots} \xi^{uv}(uv) &= \xi^u(u) (u \cdot \xi^v(v))
\end{align}
\end{lem}

\begin{lem} \label{L:xiinverse} Suppose $v,w\in W$. Then
\begin{align}\label{E:xiinverse}
\xi^v(w) = (-1)^{\ell(v)} w\cdot (\xi^{v^{-1}}(w^{-1})).
\end{align}
\end{lem}

\subsection{Peterson subalgebra and Schubert homology basis}
The Peterson subalgebra of $\A$ is the centralizer subalgebra $\B
=Z_\A(S)$ of $S$ in $\A$.

\begin{thm} \label{T:Peterson} \cite{P} There is an isomorphism
$H_T(\Gr_G) \to \B$ of commutative Hopf algebras over $S$. For $w\in
W_\af^0$ let $j_w$ denote the image of $\xi_w$ in $\B$. Then $j_w$
is the unique element of $\B$ with the property that $j_w^w=1$ and
$j_w^x=0$ for any $x\in W_\af^0 \setminus \{w\}$ where $j_w^x\in S$
are defined by
\begin{align} \label{E:jcoef}
  j_w = \sum_{x\in W_\af} j^x_w A_x.
\end{align}
Moreover, if $j^x_w\ne 0$ then $\ell(x) \ge \ell(w)$
and $j^x_w$ is a polynomial of degree $\ell(x)-\ell(w)$.
\end{thm}

The Schubert structure constants for $H_T(\Gr_G)$ are obtained as
coefficients of the elements $j_w$.

\begin{prop} \label{P:constants} \cite{P} Let $u,v,w\in W_\af^0$. Then
\begin{align}
  d^w_{uv} = \begin{cases}
  j_u^{wv^{-1}} & \text{if $\ell(w)=\ell(v)+\ell(wv^{-1})$} \\
  0 & \text{otherwise.}
  \end{cases}
\end{align}
\end{prop}

Due to the fact \cite{LS:Q} \cite{P} that the collections of Schubert
structure constants for $H_T(\Gr_G)$ and $QH^T(G/B)$ are the same and Mihalcea's positivity theorem for equivariant quantum Schubert structure constants,
we have the positivity property

\begin{prop} \label{P:positive} 
$j_w^x \in \Z_{\ge0}[\alpha_i\mid i\in I]$ for all $w\in W_\af^0$
and $x\in W_\af$.
\end{prop}

Given $u\in W_\af^0$ let $t^u=t_\la$ where $\la\in Q^\vee$ is such that $t_\la W = u W$.

Since the translation elements act trivially on $S$
and $W_\af\subset \A$ via \eqref{E:rinA},
we have $t_\la \in \B$ for all $\la \in Q^\vee$, so that $t_\la \in \bigoplus_{v\in W_\af^0} S j_v$.
For any $w\in W_\af^0$, we have
\begin{align*}
  t^w &= \sum_{\substack{w \in W_\af^0}} (-1)^{\ell(v)} \xi^v(t^w) j_v = \sum_{\substack{w \in W_\af^0}} (-1)^{\ell(v)} \xi^v(w) j_v
\end{align*}
by the definitions and Lemma \ref{L:xiprops}.

Define the $W_\af^0 \times W_\af^0$-matrices
\begin{align}
\label{E:Amatrix}
A_{wv} &= (-1)^{\ell(v)}\xi^v(w) \\
\label{E:Bmatrix}
B &= A^{-1}.
\end{align}
The matrix $A$ is lower triangular by \eqref{E:xisupport} and has nonzero diagonal terms,
and is hence invertible over $Q=\Frac(S)$. We have
\begin{align*}
  j_v = \sum_{\substack{w\in W_\af^0 \\ w \le v}} B_{wv} \,t^w.
\end{align*}
Taking the coefficient of $A_x$ for $x\in W_\af$, we have
\begin{align} \label{E:jinversion}
  j_v^x = (-1)^{\ell(x)} \sum_{\substack{w \in W_\af^0 \\ w\le v}} B_{wv}\, \xi^x(t^w).
\end{align}
Note that if $\Omega\subset W_\af^0$ is any order ideal (downwardly closed subset)
then the restriction $A|_{\Omega\times \Omega}$ is invertible.
In the sequel we choose certain such order ideals and
find a formula for the inverse of this submatrix. Since the
values of $\xi^x$ are given by \eqref{E:Billey}
we obtain an explicit formula for $j_v^x$ for $v\in \Omega$
and all $x\in W_\af$. 

\section{Equivariant Homology Chevalley rule}

\begin{thm} \label{T:jr0} For every $\id\ne x\in W_\af$,
$\xi^{x^{-1}}(r_\theta)\in \theta S$ and
\begin{align} \label{E:jr0}
  j_{r_0} = \sum_{\id\ne x\in W} (\theta^{-1}
  \xi^{x^{-1}}(r_\theta) A_x + \xi^{x^{-1}}(r_\theta) A_{r_0 x}).
\end{align}
\end{thm}
\begin{proof} For $x\ne \id$, the GKM condition \eqref{E:GKM}
and \eqref{E:xisupport} implies that
$\xi^{x^{-1}}(r_\theta)\in\theta S$. $\Omega = \{\id, r_0\} \subset W_\af^0$
is an order ideal.
The matrix $A|_{\Omega\times\Omega}$ and its inverse are given by
\begin{align*}
  \begin{pmatrix}
    1 & 0 \\
    1 & \theta
  \end{pmatrix}
\qquad
\begin{pmatrix}
 1 & 0 \\
 -\theta^{-1} & \theta^{-1}
\end{pmatrix}
\end{align*}
Since $\id = t^\id$ and $t_{\theta^\vee} = t^{r_0}$ (as $t_{\theta^\vee}=r_0r_\theta$),
we have
\begin{align*}
 (-1)^{\ell(y)} j_{r_0}^y = - \theta^{-1} \xi^y(\id) + \theta^{-1} \xi^y(t_{\theta^\vee}).
\end{align*}
By the length condition in Theorem \ref{T:Peterson} we have
\begin{align*}
  (-1)^{\ell(y)} j_{r_0}^y &= \theta^{-1} \xi^y(t_{\theta^\vee}) \qquad\text{for $y\ne \id$.}
\end{align*}
By \eqref{E:xisupport} $j_{r_0}^y = 0$ unless $y\le t_{\theta^\vee} = r_0r_\theta$.
So assume this.

Suppose $r_0y<y$. Write $y=r_0x$. Then
\begin{align*}
(-1)^{\ell(y)} \xi^y(t_\thv) &= (-1)^{\ell(y)} (\al_0) (r_0 \cdot \xi^x(r_\theta)) = (-1)^{\ell(x)} \theta (r_\theta \cdot \xi^x(r_\theta)) = \theta \, \xi^{x^{-1}}(r_\theta).
\end{align*}
If $r_0y>y$ then we write $y=x \le r_\theta$ and
\begin{align*}
(-1)^{\ell(x)}
\xi^x(t_\thv) = (-1)^{\ell(x)} r_0 \cdot \xi^x(r_\theta) = (-1)^{\ell(x)} r_\theta \cdot \xi^x(r_\theta) = \xi^{x^{-1}}(r_\theta)
\end{align*}
as required.
\end{proof}

The formula \eqref{E:Billey} shows that $\xi^{x^{-1}}(r_\theta) \in \Z_{\ge0}[\alpha_i\mid i\in I]$.  The same holds for $\theta^{-1} \xi^{x^{-1}}(r_\theta)$.  Indeed,
\begin{lem}
 $\alpha^{-1} \xi^x(r_\alpha) \in \Z_{\ge0}[\alpha_i\mid i\in I]$ for any positive root $\alpha$.  
\end{lem}
\begin{proof}
The reflection $r_\alpha$ has a reduced word $\i = i_1 i_2 \cdots i_{r-1} i_r i_{r-1} \cdots i_1$ which is symmetric.  Consider the different embeddings $\j$ of reduced words of $x$ into $\i$, as in \eqref{E:Billey}.  If $\j$ uses the letter $i_r$, then the corresponding term in \eqref{E:Billey} has $\theta$ as a factor.  Otherwise, $\j$ uses $i_s$ but not $i_{s+1}, \ldots, i_r$, for some $s$.  But then there is another embedding of $\j'$ of the same reduced word of $x$ into $\i$, which uses the other copy of the letter $i_s$ in $\i$.  The two terms in \eqref{E:Billey} which correspond to $\j$ and $\j'$ contribute $A(\beta - r_\alpha\cdot\beta) = A(\ip{\alpha^\vee}{\beta}\alpha)$ where $A \in \Z_{\ge0}[\alpha_i\mid i\in I]$, and $\beta$ is an inversion of $r_\alpha$.  It follows that $\ip{\alpha^\vee}{\beta} >0$.  The lemma follows.
\end{proof}

\begin{rem} The polynomials $\xi^{x^{-1}}(r_\theta)$ appearing in
\eqref{E:jr0} may be computed entirely in the finite Weyl group and finite weight lattice.
\end{rem}

\begin{rem}
In \cite[Proposition 2.17]{LS:DGG}, we gave an expression for the non-equivariant part of $j_{r_0}$, consisting of the terms $j_{r_0}^x A_x$ where $\ell(x) = 1 = \ell(r_0)$.  This follows easily from Theorem \ref{T:jr0} and the fact \cite{KK:H} that $\xi^{r_i}(w) = \omega_i - w \cdot \omega_i$, where $\omega_i$ is the $i$-th fundamental weight.
\end{rem}

\section{Alternating equivariant Pieri rule in classical types}
We first establish some notation for $G = SL_n$, $Sp_{2n}$, and $SO_{2n+1}$.  Our root system conventions follow \cite{Kac}.

\subsection{Special classes}
We give explicit generating classes for $H_T(\Gr_G)$.

\subsubsection{$H_T(\Gr_{SL_n})$}
Define the elements
\begin{align}
\label{E:sigh}\sh_p &=  r_{p-1}\dotsm r_1 \\
\label{E:sig} \sigma_p &= r_{p-1}\dotsm r_1r_0 = \sh_p r_0
\end{align}
So $\ell(\sh_p)=p-1$ and $\ell(\sigma_p)=p$.  These elements have associated translations
\begin{align}\label{E:t}
  t_p := t^{\sigma_{p+1}} = t_{r_p\dotsm r_2 r_1\cdot\theta^\vee}
  \qquad\text{for $0\le p\le n-2$.}
\end{align}

\subsubsection{$H_T(\Gr_{Sp_{2n}})$}
For $1\le p\le 2n-1$ we define the elements $\sh_p\in W$ by
\begin{align*}
  \sh_p &= r_{p-1}\dotsm r_2 r_1 &\qquad&\text{for $1\le p\le n$} \\
  \sh_p &= r_{2n-p-1} \dotsm r_{n-2} r_{n-1} \dotsm r_2 r_1 &\qquad&\text{for $n+1 \le p \le
  2n-1$.}
\end{align*}
For $1\le p\le 2n-1$ define $\sigma_p\in W_\af^0$ and
$t_{p-1}\in W_\af$ by
\begin{align}
\sigma_p &= \sh_p r_0 \\
t_{p-1} &= t^{\sigma_p} = t_{\sh_p\cdot \theta^\vee}.
\end{align}


\subsubsection{$H_T(\Gr_{SO_{2n+1}})$}
For $1\le p\le 2n-1$ we define the elements $\sh_p\in W_\af^0$ by
\begin{align*}
  \sh_p = \begin{cases}
  \id & \text{if $p=1$} \\
  r_p r_{p-1} \dotsm r_3 r_2 & \text{if $2\le p\le n$} \\
  r_{2n-p} r_{2n-p+1} \dotsm r_{n-1} r_n r_{n-1} \dotsm r_3 r_2 & \text{if $n+1\le p \le 2n-2$} \\
  r_0 r_2 r_3 \dotsm r_{n-1} r_n r_{n-1} \dotsm r_3 r_2 & \text{if $p=2n-1$.}
  \end{cases}
\end{align*}
For $1\le p\le 2n-1$ define $\sigma_p\in W_\af^0$ by
\begin{align}
\sigma_p &= \sh_p r_0
\end{align}
For $1\le p\le 2n-2$ define $t_{p-1}\in W_\af$ by
\begin{align}
t_{p-1} = t^{\sigma_p} =t_{\sh_p\cdot \theta^\vee}.
\end{align}

For $1\le p\le 2n-1$ let $\sigma'_p$ be $\sigma_p$ but with
every $r_0$ replaced by $r_1$.  Then define
\begin{align*}
  t_{2n-2} =t_{2\omega_1^\vee}=\sigma_{2n-1} \sigma'_{2n-1}.
\end{align*}

Then we conjecture that
\begin{align}\label{E:oddball}
B_{\sigma_{2n-1},\sigma_q} = \pm \dfrac{1}{  \xi^{\sigma_{2n-1}}(\sigma_q' \sigma_{2n-1})} \qquad\text{for $1\le q\le 2n-1$}
\end{align}
where $B$ is defined in \eqref{E:Bmatrix}. The sign is $-$ for $q\le 2n-2$ and $+$ for $q=2n-1$.

\subsubsection{Special classes generate}
Let $k'=n-1$ for $G=SL_n$ and $k'=2n-1$ for $G=Sp_{2n}$ or $G=SO_{2n+1}$.
Let $\hat{\B}:=S[[j_{\sigma_m}\mid 1\le m\le k']]$ be the completion of $\B\cong H_T(\Gr_G)$
generated over $S$ by series in the special classes.
It inherits the Hopf structure from $\B$. The Hopf structure on $\B$
is determined by the coproduct on the special classes.

\begin{prop} \label{P:generate} For $G=SL_n, Sp_{2n}, SO_{2n+1}$, 
$\Q\otimes_\Z \B \subset \Q\otimes_\Z \hat{\B}$.
\end{prop}
\begin{proof} 
It is known that the special classes generate the homology $H_*(\Gr_G)$ non-equivariantly for $G=SL_n, Sp_{2n}, SO_{2n+1}$ see \cite{LSS}\cite{Pon}.  Furthermore,  the equivariant homology Schubert structure constants $d_{uv}^w$ is a polynomial in the simple roots of degree $\ell(w)-\ell(u)-\ell(v)$, and when $\ell(w) = \ell(u)+\ell(v)$, it is equal to the non-equivariant  homology Schubert structure constant.  It follows easily from this that each equivariant Schubert class can be expressed as a formal power series in the equivariant special classes.
\end{proof}

\begin{rem} For $G=SL_n$ and $G=Sp_{2n}$ the special classes generate $H_*(\Gr_G)$ over $\Z$.
\end{rem}

\subsection{The alternating equivariant affine Pieri rule}
Let $k = n-1$ for $G = SL_n$, $k=2n-1$ for $G=Sp_{2n}$, and $k=2n-2$ for $G = SO_{2n+1}$.
Our goal is to compute $j_{\sigma_m}^x$ for $1\le m\le k$; note that for $G=SO_{2n+1}$, the element $\sigma_{2n-1}$
has been treated in \eqref{E:oddball}. For this purpose
consider the Bruhat order ideal $\Omega = \{ \id = \sigma_0, \sigma_1,\dotsc,\sigma_{k} \}$ in $W_\af^0$. 
Since $j_0 = \id$, to compute $j_{\sigma_p}^x$ for $p\ge1$
we may assume $x\ne \id$ by length considerations. It suffices to
invert the matrix $A$ given in \eqref{E:Amatrix} over $\Omega\setminus\{\id\} \times \Omega\setminus\{\id\}$.

Define the matrices $M_{pm}= (-1)^m \xi^{\sigma_m}(\sigma_p)$ for $1\le p,m\le k$,
$N_{mq}=\xi^{\sh_m r_\theta}(\sh_q r_\theta)$ for $1\le m,q\le k$, and 
the diagonal matrix $D_{pq} = \delta_{pq} \,\xi^{t_{p-1}}(t_{p-1})$ for $1\le p,q\le k$.

\begin{conj} \label{C:matrix}
\begin{align}\label{E:matrix}
  MN = D.
\end{align}
\end{conj}

\begin{conj}\label{C:altPieri}
For $1\le m\le k$ and $x\ne \id$ we have
\begin{align} \label{E:equivPieri}
  j_{\sigma_m}^x = (-1)^{\ell(x)} \sum_{q=0}^{m-1} \dfrac{\xi^{\sh_m
  r_\theta}(\sh_{q+1}
  r_\theta)}{\xi^{t_q}(t_q)} \xi^x(t_q).
\end{align}
In particular $j_{\sigma_m}^x=0$ unless $\ell(x) \ge m$ and $x \le t_q$ for some $0\le q\le m-1$.
\end{conj}

Conjecture \ref{C:altPieri} follows
immediately from Conjecture \ref{C:matrix}: we have $M^{-1}=ND^{-1}$,
and \eqref{E:equivPieri} follows from \eqref{E:jinversion}.

\begin{thm} \label{T:altPieri} Conjecture \ref{C:altPieri} holds for $G=SL_n$.
\end{thm}

The proof appears in Appendix \ref{S:proof}.
Examples of \eqref{E:matrix} appear in Appendix \ref{S:matrixexamples}.

\section{Effective Pieri rule for $H_T(\Gr_{SL_n})$}
The goal of this section is to prove a formula for $j_{\sigma_m}^x$
that is manifestly positive. In this section we work with $G=SL_n$, $W=S_n$, and $W_\af=\tS_n$.

\subsection{Simplifying \eqref{E:equivPieri}}
We first establish some notation. For $a\le b$ write
\begin{align}
u_a^b &= r_ar_{a+1}\dotsm r_b \\
d_a^b &= r_br_{b-1}\dotsm r_a \\
\al_a^b &=\al_a+\al_{a+1}+\dotsm+\al_b
\end{align}
for upward and downward sequences of reflections
and for sums of consecutive roots.
In particular we have $\theta = \alpha_1+\alpha_2+\dotsm+\alpha_{n-1} = \alpha_1^{n-1}$.

Let $0\le q\le m-1$. We have
\begin{align*}
  \xi^{\sh_m r_\theta}(\sh_{q+1} r_\theta) &=
  u_{q+1}^{m-1} \cdot \xi^{\sh_m r_\theta}(\sh_m r_\theta) \\
  &= (-1)^m u_{q+1}^{m-1}\sh_m r_\theta \cdot \xi^{r_\theta
  \sh_m^{-1}}(r_\theta \sh_m^{-1}) \\
  &=(-1)^m \sh_{q+1} r_\theta \cdot \xi^{r_\theta
  \sh_m^{-1}}(r_\theta \sh_m^{-1}) .
\end{align*}
We also have
\begin{align*}
  \xi^{t_q}(t_q) &= \xi^{\sigma_{q+1}}(\sigma_{q+1})
  (\sigma_{q+1} \cdot \xi^{r_\theta \sh_m^{-1}}(r_\theta \sh_m^{-1}))
  (\sigma_{q+1} r_\theta \sh_m^{-1} \cdot \xi^{d^{m-1}_{q+1}}(d^{m-1}_{q+1})) \\
  &=\xi^{\sigma_{q+1}}(\sigma_{q+1}) (\sh_{q+1} r_\theta \cdot \xi^{r_\theta \sh_m^{-1}}(r_\theta \sh_m^{-1}))
  (u_{q+1}^{m-1}\cdot \xi^{d^{m-1}_{q+1}}(d^{m-1}_{q+1})) \\
  &= (-1)^{m-q-1} \xi^{\sigma_{q+1}}(\sigma_{q+1}) (\sh_{q+1} r_\theta \cdot \xi^{r_\theta \sh_m^{-1}}(r_\theta \sh_m^{-1}))
  \xi^{u_{q+1}^{m-1}}(u_{q+1}^{m-1}).
\end{align*}
Define
$$
D(q,m) = \xi^{\sigma_{q+1}}(\sigma_{q+1}) \xi^{u_{q+1}^{m-1}}(u_{q+1}^{m-1}).
$$
so that by Theorem \ref{T:altPieri},
\begin{align} \label{E:jsum}
  j_{\sigma_m}^x = (-1)^{\ell(x)} \sum_{q=0}^{m-1} \dfrac{(-1)^{q+1}}{D(q,m)} \xi^x(t_q).
\end{align}
Explicitly we have
\begin{align}
\label{E:xisigma}
  \xi^{\sigma_{q+1}}(\sigma_{q+1}) &= \al_q \al_{q-1}^q \dotsm \al_1^q \al_0^q \\
\label{E:xiu}
  \xi^{u_{q+1}^{m-1}}(u_{q+1}^{m-1}) &= \al_{q+1} \al_{q+1}^{q+2}\dotsm \al_{q+1}^{m-1}.
\end{align}

\subsection{$V$'s and $\Lambda$'s}
The support $\Supp(b)$ of a word $b$ is the set of letters appearing
in the word. For a permutation $w$, $\Supp(w)$ is the support of any
reduced word of $w$. A $V$ is a reduced word (for some permutation)
that decreases to a minimum and increases thereafter. Special cases
of $V$'s include the empty word, any increasing word and any
decreasing word. A $\Lambda$ is a reduced word that increases to a
maximum and decreases thereafter. A (reverse) $\rN$ is a reduced
word consisting of a $V$ followed by a $\Lambda$, such that the
support of the $V$ is less than the support of the $\Lambda$.

By abuse of language, we say a permutation is a $V$ if it admits a
reduced word that is a $V$. We use similar terminology for
$\Lambda$'s and $\rN$'s.

A permutation is connected if its support is connected (that is, is a
subinterval of the integers).

\begin{lem} \label{L:uniqueword}
A permutation that is a $V$, admits a unique reduced word that is a
$V$. Similarly for a connected $\Lambda$ or a connected $\rN$.
\end{lem}

\begin{lem} \label{L:equivalentforms}
A connected permutation is a $V$ if and only if it is
a $\Lambda$, if and only if it is an $\rN$.
\end{lem}


\subsection{$t_q$-factorizations} 
For $0\le q\le n-2$, we call
\begin{align} \label{E:stdtq}
  q (q-1) \dotsm 1 0 1 \dotsm (n-1) (n-2) \dotsm q+1
\end{align}
the standard reduced word for $t_q$. Since this word is an $\rN$
it follows that any $x\le t_q$ is an $\rN$.  We call the subwords $q(q-1)\dotsm 1$, $12\dotsm (n-2)$ and $(n-2) \dotsm q+1$ the left, middle, and right branches.

\begin{lem} \label{L:embedt}
If $x\in \tS_n$ admits a reduced word in which $i+1$
precedes $i$ for some $i\in \Z/n\Z$ then $x\not\le t_i$.
\end{lem}
\begin{proof} Suppose $x\le t_i$. Since the standard reduced word of
$t_i$ has all occurrences of $i$ preceding all occurrences of $i+1$,
it follows that $x$ has a reduced word with that property.  But this property is invariant under the braid
relation and the commuting relation, which connect all reduced words of $x$.
\end{proof}

Let $c(x)$ denote the number of connected components of $\Supp(x)$.
If $J$ and $J'$ are subsets of integers then we write $J<J'-1$ if
$\max(J)<\min(J')-1$.

\begin{lem} \label{L:qfact}
Suppose $x \leq t_q$.  Then $x$ has a unique factorization $x = v_1
\cdots v_r y_1 y_2 \cdots y_s$, called the {\it $q$-factorization},
where each $v_i,y_i$ has connected support such that
\begin{enumerate}
\item
$\Supp(v_i) < \Supp(v_{i+1})-1$ and $\Supp(y_i) < \Supp(y_{i+1})-1$
\item
$\Supp(v_1 \cdots v_r) \subset [0,q]$
\item
$\Supp(y_1 \cdots y_s) \subset [q+1,n-1]$
\item
Each $v_i$ is a $V$.
\item
Each $y_i$ is a $\Lambda$.
\end{enumerate}
\end{lem}

We say that $v_r$ and $y_1$ {\it touch} if $q\in \Supp(v_r)$ and $q+1 \in \Supp(y_1)$.  We denote
\begin{align}\label{E:epsdef}
  \epsilon(x,q) = \begin{cases}
  1 & \text{if $v_r$ and $y_1$ touch} \\
  0 & \text{otherwise.}
  \end{cases}
\end{align}
Note that $\epsilon(x,q)$ depends only on $\Supp(x)$ and $q$.

Each $k$ in the $q$-factorization of $x\le t_q$, is (1) in the left
branch of some $v_i$, or (2) in the right branch of some $v_i$, or
(3) at the bottom of a $v_i$, or (1') in the left branch of some
$y_i$, or (2') in the right branch of some $y_i$, or (3') at the top
of a $y_i$, or finally (4) absent. We call these sets $S1$, $S2$,
$S3$, $S1'$, $S2'$, and $S3'$.  Note that $k$ can belong to both
$S1$ and $S2$, or both $S1'$ and $S2'$.

For each $x$ and each $q$ such that $x \leq t_q$, we define the polynomials
\begin{align*}
M(x,q) &= (\alpha_0^q)^{\epsilon(x,q)}\prod_{k \in S2} \alpha_0^{k-1} \; \prod_{k \in S1'} \alpha_0^k \\
L(x,q) &= \prod_{k \in S1} \alpha_k^q \\
R(x,q) &= \prod_{k\in S2'} (-\alpha_{q+1}^k)
\end{align*}
We also define $R(x,q,m)= \prod_{k\in S2'
\cap[m,n-1]}(-\alpha_{q+1,k})$.

\begin{prop} \label{P:xieval}
If $x \leq t_q$, then
\begin{align} \label{E:xieval}
 \xi^x(t_q) = (\alpha_0^q)^{c(x)}\; M(x,q)\;L(x,q)\;R(x,q).
\end{align}
\end{prop}
\begin{proof}We compute
$\xi^x(t_q)$ using \eqref{E:Billey} by computing all embeddings of reduced words of $x$ into the standard reduced word \eqref{E:stdtq} of $t_q$. We refer to the $q$-factorization
of $x$. Each $k\in S1$ must embed into the left branch of the $\rN$,
and has associated root $\alpha_k^q$. Each $k\in S2$ embeds into
the middle branch of the $\rN$ and has associated root
$\alpha_0^{k-1}$. Each $k\in S1'$ embeds into the middle branch of
the $\rN$ and has associated root $\alpha_0^k$. Each $k\in S2'$
embeds into the right branch of the $\rN$ and has associated root
$-\alpha_{q+1}^k$. Each $k\in S3$ is either $0$ and has associated
root $\alpha_0^q$, or can be embedded into the left or middle
branch of the $\rN$, and the sum of the two associated roots for
these positions is $\alpha_k^q+\alpha_0^{k-1}=\alpha_0^q$. Each
$k\in S3'$ is either $n-1$, which has associated root
$-\alpha_{q+1}^{n-1} = \alpha_0^q$, or can be embedded into the
middle or right branch of the $\rN$, and the sum of associated roots
is $\alpha_0^k-\alpha_{q+1}^k=\alpha_0^q$. Since all the various
choices for embeddings of elements of $S3$ and $S3'$ can be varied
independently, the value of $\xi^x(t_q)$ is the product of the above
contributions. Each minimum of a $v_i$ and maximum of a $y_j$
contributes $\alpha_0^q$. If there is a component of $x$ which
contains both $q$ and $q+1$ (that is, if $v_r$ and $y_1$ touch) then
it is unique and contributes two copies of $\alpha_0^q$. All this
yields \eqref{E:xieval}.
\end{proof}

\subsection{Rotations}
We now relate $\xi^x(t_q)$ with $\xi^x(t_{q'})$. Let $r_p^q$ denote the
transposition that exchanges the integers $p$ and $q$.

\begin{prop}\label{P:rotate}
Let $x \leq t_q$ and consider the $q$-factorization of $x$.  Let $a$ be such that
this reduced word of $x$ contains the decreasing subword $(q+a)(q+a-1)\dotsm (q+1)$ but not
$(q+a+1)(q+a)\dotsm(q+1)$.  If $q+1 \notin \Supp(x)$, then set $a = 1$.  Then
\begin{align} \label{E:betweenvanish}
 \xi^x(t_{q+1}) = \xi^x(t_{q+2}) = \cdots = \xi^x(t_{q+a-1})=0
\end{align}
and
\begin{align}
\xi^x(t_{q+a}) &= M(x,q) \, r_{1+q}^{1+q+a} (\alpha_0^q)^{c(x)}\; L(x,q)\;R(x,q) \label{E:tchangefin}
\end{align}
\end{prop}

Let $y^{\uparrow}$ denote $y$ with every $r_i$ changed to $r_{i+1}$.

\begin{lem}\label{L:commute}
Let $y$ be increasing with support in $[b,a-1]$.  Then
$$
y d^a_b = d^a_b y^\uparrow
$$
\end{lem}

\begin{proof}[Proof of Proposition \ref{P:rotate}]
We assume that $q +1 \in \Supp(x)$, for otherwise the claim is easy.

By Lemma \ref{L:embedt} we have $x\not\le t_{q+i}$ for $1\le i\le
a-1$. Equation \eqref{E:betweenvanish} follows from
\eqref{E:xisupport}. We now prove \eqref{E:tchangefin}. The first goal
is to compute the $q+a$-factorization of $x$. Since $x\le t_q$
we may consider the $q$-factorization of $x$. The decreasing word
$(q+a-1)\dotsm(q+2)(q+1)$ must embed into the right hand branch,
that is, $[q+1,q+a-1]\subset S2'$. The hypotheses imply
that $q+a\not\in S2'$. There are two cases: either $q+a\in S1'$ or
$q+a\in S3'$ (so that $q+a+1\not\in\Supp(x)$). We treat the former
case, as the latter is similar: the two cases correspond to the
touching and nontouching cases for the $q+a$-factorization of $x$,
whose existence we now demonstrate.

Suppose $q+a\in S1'$. Then there is a $y_1'$ with
$\Supp(y_1')\subset [q+a+1,n-1]$ and a $y$ with an increasing
reduced word such that $\Supp(y)\subset [q+1,q+a-1]$ and
$y_1 = y r_{q+a} y_1' d^{q+a-1}_{q+1}
= y d^{q+a}_{q+1} y_1'$.
Suppose $v_r$ and $y_1$ touch. Then $v_r':=v_r y d^{q+a}_{q+1}$
is an $\rN$ and therefore a $V$. Moreover
$x\le t_{q+a}$ since $x$ has a $q+a$-factorization given by the
$q$-factorization of $x$ but with $v_r$ and $y_1$ replaced by $v_r'$
and $y_1'$ respectively. To verify that $v_r'$ is a $V$, by the
touching assumption, $q\in\Supp(v_r)$ and we have
$v_r' = v_r y d^{q+a}_{q+1} = v_r d^{q+a}_{q+1} y^\uparrow = d^{q+a}_{q+2} v_r r_{q+1} y^\uparrow$
which expresses $v_r'$ in a $V$.

Suppose $v_r$ and $y_1$ do not touch, that is, $q\notin\Supp(v_r)$.
We have the $V$ given by $v'_{r+1} = y d^{q+a}_{q+1} = d^{q+a}_{q+1} y^\uparrow$.
Then $x\le t_{q+a}$, as $x$ has the $q+a$ factorization given by the
$q$-factorization of $x$ except that there is a new $V$, namely,
$v'_{r+1}$ and the first $y$ is $y_1'$ instead of $y_1$.

In every case we calculate that
\begin{align*}
M(x,q+a) &= M(x,q)\\
L(x,q+a) &= \left(\prod_{k=q+2}^{q+a} \alpha_k^{q+a}\right) d^{q+a}_{q+1} L(x,q) \\
R(x,q+a)&= d^{q+a}_{q+1}\left(\prod_{k=q+1}^{q+a-1} (-\alpha_{q+1}^k)^{-1} \right) R(x,q) = \left(\prod_{k=q+1}^{q+a-1} (\alpha_k^{q+a})^{-1}\right) d^{q+a}_{q+1} R(x,q).
\end{align*}
The calculation for $L$ and $R$ follows from the fact that
$[q+2,q+a]\subset S1_{q+a}$, but $[q+1,q+a-1]\subset S2'_q$. The
calculation for $M$ follows from the fact that $\Supp(y) \subset
S2_q$ and $\Supp(y^{\uparrow}) \subset S2_{q+a}$, together with the
following boundary cases:

If $q+a+1 \in \Supp(x)$ then $q+a \in S1_{q+a}\cap S1'_q$.  Thus
$q+a$ contributes a factor of $\alpha_0^{q+a}$ to $M(x,q)$.  This
factor appears in $M(x,q+a)$ as the factor
$(\alpha_0^{q+a})^{\epsilon(x,q+a)}$, since $\epsilon(x,q+a) = 1$.

If $q \in \Supp(x)$ one has $\epsilon(x,q) = 1$ and $q+1 \in
S2_{q+a}$ contributes a factor of $\alpha_0^q$ to $M(x,q+a)$. This
factor appears in $M(x,q)$ as the factor
$(\alpha_0^q)^{\epsilon(x,q)} = \alpha_0^q$.

Using that $d^{q+a}_{q+1} \al_0^q = \al_0^{q+a}$,
$d^{q+a}_{q+1}(-\alpha_{q+1}^{q+a})=\alpha_{q+a}$, and $r_{1+q}^{1+q+a}
\al_{q+1}^{q+a} = -\al_{q+1}^{q+a}$, the above relations between
$M(x,q)$, $L(x,q)$, $R(x,q)$ and their counterparts for $q+a$,
together with Proposition \ref{P:xieval}, yield
$$\xi^x(t_{q+a})= (\alpha_{q+1}^{q+a})^{-1} M(x,q)\, d^{q+a}_{q+1} \,(-\alpha_{q+1}^{q+a}) (\alpha_0^q)^{c(x)}\; L(x,q)\;R(x,q).$$
To obtain \eqref{E:tchangefin},
since $r_{1+q}^{1+q+a} = d^{q+a}_{q+1}
u_{q+2}^{q+a}$, it suffices to show that
\begin{align*}
\text{$(-\alpha_{q+1}^{q+a}) (\alpha_0^q)^{c(x)} L(x,q) R(x,q)$ is invariant under $u_{q+2}^{q+a}$.}
\end{align*}
However it is clear that $\alpha_0^q$ and $L(x,q)$ are invariant,
and the only part of $R(x,q)$ that must be checked is the product
$\prod_{k\in S2' \cap [q+1,q+a]} (-\alpha_{q+1,k})$. However we have
that $S2'\cap [q+1,q+a] = [q+1,q+a-1]$, and indeed the product
$\prod_{k=q+1}^{q+a} (-\alpha_{q+1}^k)$ is invariant under
$u_{q+2}^{q+a}$, as required.
\end{proof}

Let
\begin{align} \label{E:qset}
  \{ q\in [0,m-1] \mid x \le t_q \} = \{ q_1 < q_2 <\dotsm < q_p\}.
\end{align}
In light of the proof of Proposition \ref{P:rotate}, we write
\begin{align} \label{E:Mdef}
M(x)=M(x,q_j)\qquad\text{for any $1\le j\le p$.}
\end{align}
Let
$\beta_i = \alpha_{1+q_i}^{q_{i+1}}$
be the root associated with the reflection
$r_{\beta_i}$ that exchanges the numbers $1+q_i$ and $1+q_{i+1}$.
For $i\le j$ we also define
\begin{align*}
  \beta_i^j = \beta_i+\beta_{i+1}+\dotsm+\beta_j = \alpha_{q_i+1}^{q_{j+1}}.
\end{align*}
Let
\begin{align}\label{E:Ydef}
Y_i(x,m) = (\alpha_0^{q_i})^{c(x)-1}R(x,q_i,m)\qquad\text{for $1\le
i\le p$}
\end{align}
so that $Y_i(x,m) = r_{\beta_{i-1}} Y_{i-1}(x,m)$.

\begin{lem}
\label{L:xiprime}
$$
(-1)^{m-1-q_j-p+j}
\frac{\xi^x(t_{q_j})}{D(q_j,m)}= \frac{M(x)Y_j(x,m)}{(\beta_1^{j-1}\beta_2^{j-1}\dotsm\beta_{j-1}^{j-1})(\beta_j^j\beta_j^{j+1}\dotsm\beta_j^{p-1})}.
$$
\end{lem}
\begin{proof} The proof proceeds by induction on $j$. Let $D_j$ be the denominator of the
right hand side. Suppose first that $j=1$.
Consider the embedding of $x$ into $t_{q_1}$. By the definition of $q_1$,
it follows that $L(x,q_1) \al_0^{q_1} = \xi^{\sigma_{q_1+1}}(\sigma_{q_1+1})$.
By the definition of the $q_j$, we also have $S2' \cap [q_1+1,m-1] = [q_1+1,m-1] \setminus \{q_2,q_3,\dotsc,q_p\}$.
These considerations and Proposition \ref{P:xieval} imply that
\begin{align*}
\xi^x(t_{q_1}) &= (\alpha_0^{q_1})^{c(x)} M(x) L(x,q_1) R(x,q_1) \\
&= (-1)^{m-1-q_1}(\alpha_0^{q_1})^{c(x)} M(x) D(q_1,m) R(x,q_1,m) \prod_{j=2}^p (-\alpha_{q_1+1}^{q_j})^{-1} \\
&= (-1)^{m-1-q_1-p+1} D(q_1,m) M(x) Y_1(x,m)  D_1^{-1}.
\end{align*}
This proves the result for $j=1$.
Suppose the result holds for $1\le j\le p-1$. We show it holds for $j+1$.  By induction we have
\begin{align*}
(\alpha_0^{q_j})^{c(x)} L(x,q_j) R(x,q_j) = \dfrac{D(q_j,m) Y_j(x,m)}{D_j}.
\end{align*}
Proposition \ref{P:rotate} yields
\begin{align*}
  \dfrac{\xi^x(t_{q_{j+1}})}{D(q_{j+1},m)} &=
  \dfrac{M(x) r_{\beta_j} (\alpha_0^{q_j})^{c(x)} L(x,q_j) R(x,q_j)}{D(q_{j+1},m)} \\
&=  \dfrac{M(x)}{D(q_{j+1},m)} r_{\beta_j} \dfrac{D(q_j,m)Y_j(x,m)}{D_j} \\
&= \dfrac{M(x)Y_{j+1}(x,m)}{D(q_{j+1},m)} r_{\beta_j} \dfrac{D(q_j,m)}{D_j}.
\end{align*}
It remains to show
\begin{align*}
  (-1)^{q_{j+1}-q_j-1} \dfrac{D(q_{j+1},m)}{ D_{j+1}} = r_{\beta_j} \dfrac{D(q_j,m)}{D_j}.
\end{align*}
We have $D(q_j,m)=\prod_{k=0}^{q_j} \alpha_k^{q_j} \prod_{k=q_j+1}^{m-1} \alpha_{q_j+1}^k$.
For $k\in[0,q_j]$ we have $r_{\beta_j} \alpha_k^{q_j} = \alpha_k^{q_{j+1}}$.
For $k\in [q_j+1,q_{j+1}-1]$ we have $r_{\beta_j} \alpha_{q_j+1}^k = - \alpha_{k+1}^{q_{j+1}}$,
$r_{\beta_j} \alpha_{q_j+1}^{q_{j+1}} = -\alpha_{q_j+1}^{q_{j+1}}$,
and for $k\in [q_{j+1}+1,m-1]$ we have $r_{\beta_j} \alpha_{q_j+1}^k = \alpha_{q_{j+1}+1}^k$.
Therefore
\begin{align*}
  r_{\beta_j} D(q_j,m) &= (-1)^{q_{j+1}-q_j} \prod_{k=0}^{q_j} \alpha_k^{q_{j+1}} \prod_{k=q_j}^{q_{j+1}-1} \alpha_{k+1}^{q_{j+1}}
\prod_{k=q_{j+1}+1}^{m-1} \alpha_{q_{j+1}+1}^k\\
&= (-1)^{q_{j+1}-q_j} D(q_{j+1},m).
\end{align*}
We also have $r_{\beta_j} \beta_{j-1}^i = \beta_j^i$ for $1\le i\le j-1$ and
$r_{\beta_j} \beta_j^i = \beta_{j+1}^i$ for $j+1 \le i\le p-1$. Therefore
\begin{align*}
  r_{\beta_j} D_j 
= \left(\prod_{i=1}^{j-1} \beta_i^j\right)
  (-\beta_j)\left(\prod_{i=j+1}^{p-1} \beta_{j+1}^i\right) = - D_{j+1}.
\end{align*}
\end{proof}

\begin{lem}
$r_{\beta_j} Y_i(x,m) = Y_i(x,m)$ for $j\ge i+2$.
\end{lem}

\subsection{The equivariant Pieri rule}

For a root $\beta$ and $f\in S$ define
\begin{align*}
  \partial_\beta f = \beta^{-1}(f - r_\beta f).
\end{align*}

\begin{thm} \label{T:Pieri}
$$
j_{\sigma_m}^x = (-1)^{\ell(x)-m+p-1} M(x) \partial_{\beta_{p-1}} \dotsm \partial_{\beta_2} \partial_{\beta_1} Y(x,m)
$$
where $Y(x,m) = Y_1(x,m)$.
\end{thm}

\begin{proof}
Note that if $r_{\beta_{j+1}} Y = Y$ and $i\le j$ then
$$
\frac{1}{\beta_{j+1}}(1-r_{\beta_{j+1}})\frac{Y}{\beta_i^j} = \frac{Y}{\beta_i^j\beta_i^{j+1}}.
$$

\begin{align*}
&\partial_{\beta_{p-1}} \dotsm \partial_{\beta_2} \partial_{\beta_1} Y(x,m) \\
&= \frac{1}{\beta_{p-1}}(1-r_{\beta_{p-1}}) \cdots \frac{1}{\beta_1}(1-r_{\beta_1})Y_1(x,m)\\
&= \frac{1}{\beta_{p-1}}(1-r_{\beta_{p-1}}) \cdots \frac{1}{\beta_2}(1-r_{\beta_2})(\frac{Y_1(x,m)}{\beta_1} - \frac{Y_2(x,m)}{\beta_1}) \\
&=\frac{1}{\beta_{p-1}}(1-r_{\beta_{p-1}}) \cdots \frac{1}{\beta_3}(1-r_{\beta_3})
\left(\frac{Y_1(x,m)}{\beta_1\beta_1^2} - \frac{Y_2(x,m)}{\beta_1\beta_2} + \frac{Y_3(x,m)}{\beta_1^2\beta_2} \right)\\
&=\cdots\\
&= \frac{Y_1(x,m)}{\beta_1\beta_1^2\dotsm \beta_1^{p-1}}
- \frac{Y_2(x,m)}{\beta_1 \beta_2 \beta_2^3 \dotsm \beta_2^{p-1}}
+\dotsm+(-1)^j\frac{Y_{j+1}(x,m)}{\beta_1^j \dotsm \beta_{j-1}^j \beta_j \beta_{j+1} \beta_{j+1}^{j+2} \dotsm \beta_{j+1}^{p-1}} \\&+ \dotsm
+(-1)^{p-1} \frac{Y_p(x,m)}{\beta_1^{p-1} \dotsm \beta_{p-2}^{p-1} \beta_{p-1}}.
\end{align*}
Thus
\begin{align*}
M(x)\partial_{\beta_{p-1}} \dotsm \partial_{\beta_2} \partial_{\beta_1} Y(x,m) &=
\sum_{j=1}^p (-1)^{j-1} \frac{M(x) Y_j(x,m)}{D_j} \\
&= (-1)^{m-p} \sum_{j=1}^p (-1)^{q_j} \frac{\xi^x(t_{q_j})}{D(q_j,m)} \\
&= (-1)^{m-p} \sum_{i=0}^{m-2} (-1)^i \frac{\xi^x(t_i)}{D(i,m)} \\
&= (-1)^{m-p+1}(-1)^{\ell(x)} j_{\sigma_m}^x
\end{align*}
by \eqref{E:jsum}, as required.
\end{proof}

\subsection{Positivity}
In this section we explain how
the expression of Theorem \ref{T:Pieri} can be written explicitly as the sum of products of positive roots.

To see this we first count the gratuitous negative signs in $M(x)=M(x,q_1)$ and $Y(x,m)=Y_1(x,m)$.
Letting $q=q_1$, using the $q_1$-factorization of $x$, and defining $\tilde{S2}'=S2'\cap[m,n-1]$,
this number is
\begin{align*}
  &\quad\;\epsilon(x,q)+|S2|+|S1'|+c(x)-1+|\tilde{S2}'| \\
  &= |S2|+|S1'|+|S3|+|S3'|-1+|\tilde{S2}'| \\
  &= \ell(x) - 1 - |S1| - |S2'\setminus \tilde{S2}'| \\
  &= \ell(x) - 1 - q_1 - |[q_1+1,m-1]\setminus \{q_2,q_3,\dotsc,q_p\}| \\
  &= \ell(x) - 1 - q_1 - (m-1-q_1-(p-1)) \\
  &= \ell(x) - m + p - 1.
\end{align*}
Therefore all signs cancel and we have that
\begin{align}\label{E:jsignless}
  j_{\sigma_m}^x = (\alpha_{q+1}^{n-1})^{\epsilon(x,q)} \prod_{k\in S2} \alpha_k^{n-1} \prod_{k\in S1'} \alpha_{k+1}^{n-1}\;
  \partial_{\beta_{p-1}} \dotsm \partial_{\beta_1} (\alpha_{q+1}^{n-1})^{c(x)-1} \prod_{k\in \tilde{S2}'} \alpha_{q+1}^k.
\end{align}
Let $x_i$ be the standard basis of the finite weight lattice $\Z^n$ with
$\alpha_i = x_i-x_{i+1}$. Then $r_{\beta_j}$ acts by exchanging $x_{q_j+1}$ and $x_{q_{j+1}+1}$.
Let us write
\begin{align*}
Z &= \alpha_{q+1}^{c(x)-1} \prod_{k\in \tilde{S2}'} \alpha_{q+1}^k = \alpha_{q+1}^{k_1} \alpha_{q+1}^{k_2} \dotsm \alpha_{q+1}^{k_d} 
= \prod_{i=1}^d (x_{q_1+1}-x_{k_i+1}).
\end{align*}
where $n-1 \ge k_1 \ge k_2 \ge \dotsm \ge k_d \ge m$. Note that $q_j+1 \le q_p+1 \le m$.
Since
$$\partial_i \cdot (fg) = (\partial_i\cdot f)g + (r_i\cdot f)(\partial_i \cdot g),$$
and since $\partial_i 1 = 0$, we have
\begin{align*}
  \partial_{\beta_1} Z &= (\partial_{\beta_1}\cdot(x_{q_1+1}-x_{k_1+1})) (x_{q_1+1}-x_{k_2+1}) \dotsm (x_{q_1+1}-x_{k_d+1}) \\
  &+ (x_{q_2+1}-x_{k_1+1}) (\partial_{\beta_1}\cdot(x_{q_1+1}-x_{k_2+1})) (x_{q_1+1}-x_{k_3+1}) \dotsm (x_{q_1+1}-x_{k_d+1}) \\
  &+ \dotsm \\
  &+ (x_{q_2+1}-x_{k_1+1}) \dotsm (x_{q_2+1}-x_{k_{d-1}}) \partial_{\beta_1}(x_{q_1+1}-x_{k_d+1}) \\
  &= \sum_{i=1}^d (x_{q_2+1}-x_{k_1+1}) \dotsm (x_{q_2+1}-x_{k_{i-1}+1}) \times \\
  &\qquad (x_{q_1+1}-x_{k_{i+1}+1}) \dotsm (x_{q_1+1}-x_{k_d+1}).
\end{align*}
So $\partial_{\beta_1}$ can act on any factor (giving the answer $1$ and thus effectively removing the factor),
and to the left each variable $x_{q_1+1}$ is reflected to $x_{q_2+1}$.
Next we apply $\partial_{\beta_2}$. It kills any factor $x_{q_1+1}-x_{k_i+1}$. Therefore we may assume
it acts on a factor of the form $x_{q_2+1}-x_{k_i+1}$ which is to the left of the factor removed by $\partial_{\beta_1}$.
Continuing in this manner we see that $\partial_{\beta_{p-1}}\dotsm \partial_{\beta_1} Z$
is the sum of products of positive roots, where a given summand corresponds to the selection of $p-1$
of the factors, which are removed, and between the $r$-th and $r+1$-th removed factor from the right,
an original factor $x_{q_1+1}-x_{k_i+1}$ is changed to $x_{q_{r+1}+1}-x_{k_i+1}$.

It follows that Theorem \ref{T:Pieri} yields a positive answer
which can be given an explicit formula, which is determined by finding the minimum $q$ for which
$x\le t_q$, deriving the quantities $c(x)$, $\epsilon(x,q_1)$, $S2$, $S1'$, $\tilde{S2}'$
(where only the last depends on $m$), and then writing down the above sum of products of positive roots.

\begin{ex} \label{X:Pieri} Let $n=8$, $m=4$, and $x=r_0r_4r_5r_7r_4r_2r_1$. The components of $\Supp(x)$
are $[0,2]$, $[4,5]$, and $[7]$. We have $p=3$ with $(q_1,q_2,q_3)=(0,2,3)$, $v_1=r_0$,
$y_1=r_2r_1$, $y_2=r_4r_5r_4$, $y_3=7$, $\epsilon(x,q_1)=1$, $S1=S2=\emptyset$, $S3=\{0\}$,
$S1'=\{4\}$, $S2'=\{1,4\}$, $S3'=\{2,5,7\}$, $S2'\cap [m,n-1]=\{4\}$. Then \eqref{E:jsignless} yields
\begin{align*}
  j_{\sigma_m}^x &= (\alpha_1^7)^1 \alpha_5^7 \partial_{\alpha_3} \partial_{\alpha_1+\alpha_2} (\alpha_1^7)^2 \alpha_1^4 \\
  &= (x_1-x_8)(x_5-x_8) \partial_{x_3-x_4} \partial_{x_1-x_3} (x_1-x_8)^2 (x_1-x_5) \\
  &= (x_1-x_8)(x_5-x_8) \partial_{x_3-x_4} ((x_1-x_8)(x_1-x_5)+(x_3-x_8)(x_1-x_5)+(x_3-x_8)^2) \\
  &= (x_1-x_8)(x_5-x_8) ((x_1-x_5)+(x_3-x_8)+(x_4-x_8)) \\
  &= (\alpha_1^7)(\alpha_5^7)(\alpha_1^4+\alpha_3^7+\alpha_4^7).
\end{align*}
\end{ex}

\appendix

\section{Proof of Theorem \ref{T:altPieri}}
\label{S:proof}

In this section we assume that $G=SL_n$ and prove \eqref{E:matrix}.

The matrices $M$ and $N$ are easily seen to be lower
triangular. We first check the diagonal:
\begin{align*}
  M_{pp} N_{pp} &= (-1)^p \xi^{\sigma_p}(\sigma_p) \xi^{\sh_p
  r_\theta}(\sh_p r_\theta) \\
  &= \xi^{\sigma_p}(\sigma_p) (\sh_p r_\theta \cdot
  \xi^{r_\theta\sh_p^{-1}}(r_\theta \sh_p^{-1})) \\
  &= \xi^{\sigma_p}(\sigma_p) (\sigma_p \cdot
  \xi^{r_\theta\sh_p^{-1}}(r_\theta \sh_p^{-1})) \\
  &= \xi^{t_{p-1}}(t_{p-1}),
\end{align*}
by \eqref{L:xiinverse}, \eqref{E:r0theta}, and Lemma \ref{L:xiprops}.

It remains to check below the diagonal. Let $p > q$ and $p \ge k \ge
q$. We have
\begin{align*}
  M_{pk} &= (-1)^k \xi^{\sigma_k}(\sigma_p) \\
  &= (-1)^k  d^{p-1}_k \cdot \xi^{\sigma_k}(\sigma_k) \\
  &= (-1)^k d^{p-1}_k \cdot (\xi^{d^{k-1}_q}(d^{k-1}_q) d^{k-1}_q\cdot
  \xi^{\sigma_q}(\sigma_q)) \\
  &= (-1)^k (d^{p-1}_k \cdot \xi^{d^{k-1}_q}(d^{k-1}_q))
  (d^{p-1}_q \cdot \xi^{\sigma_q}(\sigma_q)).
\end{align*}
Note that the second factor is independent of $k$. We also have
\begin{align*}
  N_{kq} &= \xi^{\sh_k r_\theta}(\sh_q r_\theta) \\
  &= u_q^{k-1} \cdot (\xi^{\sh_k r_\theta}(\sh_k r_\theta)) \\
  &= u_q^{k-1} \cdot (\xi^{u_k^{p-1}}(u_k^{p-1}) (u_k^{p-1}\cdot
  \xi^{\sh_p r_\theta}(\sh_p r_\theta))) \\
  &= (u_q^{k-1} \cdot \xi^{u_k^{p-1}}(u_k^{p-1}))
  (u_q^{p-1}\cdot\xi^{\sh_p r_\theta}(\sh_p r_\theta))
\end{align*}
with the second factor independent of $k$. Therefore, to prove that
$$\sum_{q\le k\le p} M_{pk} N_{kq}=0$$ it is equivalent to show that
\begin{align} \label{E:ident}
0=  \sum_{q\le k\le p} (-1)^k (d^{p-1}_k \cdot
  \xi^{d^{k-1}_q}(d^{k-1}_q)) (u_q^{k-1} \cdot
  \xi^{u_k^{p-1}}(u_k^{p-1})).
\end{align}
The above identity can be rewritten as
\begin{align} \label{E:ident2}
0=
  \sum_{q\le k\le p} (-1)^k \prod_{i=q}^{k-1} \al_i^{p-1}
  \prod_{m=k}^{p-1} \al_q^m.
\end{align}
To prove this last identity, let $q'$ be such that $q<q'\le p$. It
is easy to show by descending induction on $q'$ that
\begin{align}
  \sum_{q'\le k\le p} (-1)^k \prod_{i=q}^{k-1} \al_i^{p-1}
  \prod_{m=k}^{p-1} \al_q^m = (-1)^{q'} \prod_{i=q+1}^{q'-1} \al_i^{p-1} \prod_{m=q'-1}^{p-1}
  \al_q^m.
\end{align}
Then for $q'=q+1$ the sum is the negative of the $k=q$ summand of
\eqref{E:ident2} as required.

\section{Examples of \eqref{E:matrix}}
\label{S:matrixexamples}

\begin{ex} \label{X:SL3}
$G=SL_3$  has affine Cartan matrix
\begin{align*}
\begin{pmatrix}
2 & -1 & -1 \\
-1 & 2 & -1 \\
-1 & -1 & 2
\end{pmatrix}.
\end{align*}
The column dependencies give the coefficients of the null root
$\delta=\al_0+\theta=\al_0+\al_1+\al_2$ which is set to zero due to the
finite torus equivariance.
\begin{equation*}
\begin{array}{|c||c|c|c|c|} \hline
p & \sh_p & \sigma_p & t_{p-1} & \sh_p r_\theta \\ \hline 1 & \id & r_0
& r_0r_1r_2r_1 & r_1r_2r_1 \\
2 & r_1 & r_1r_0 & r_1r_0r_1r_2 & r_2r_1 \\ \hline
\end{array}
\end{equation*}
We compute the matrices
\begin{equation*}
M=\begin{pmatrix} \al_1+\al_2 & 0 \\ \al_2 & -\al_1\al_2
\end{pmatrix}\quad
N=\begin{pmatrix} \al_1\al_2(\al_1+\al_2) & 0 \\ \al_2(\al_1+\al_2)
& \al_2(\al_1+\al_2)
\end{pmatrix}
\end{equation*}
\begin{equation*}
D=\begin{pmatrix} \al_1\al_2(\al_1+\al_2)^2 & 0 \\ 0 &
-\al_1\al_2^2(\al_1+\al_2)
\end{pmatrix}
\end{equation*}
\begin{equation*}
ND^{-1}=\begin{pmatrix} (\al_1+\al_2)^{-1} & 0 \\
(\al_1(\al_1+\al_2))^{-1} & -(\al_1\al_2)^{-1}
\end{pmatrix}
\end{equation*}
For $x=r_1r_2$ we compute the column vector with values $(-1)^{\ell(x)}\xi^x(t_j)$
for $j=1,2$. Acting on this column vector by $ND^{-1}$, we obtain
the coefficients of $A_x$ in $j_1$ and $j_2$.
\begin{equation*}
(-1)^{\ell(x)}
\begin{pmatrix}
\xi^x(t_1) \\ \xi^x(t_2)
\end{pmatrix}=
\begin{pmatrix}
\al_2(\al_1+\al_2) \\ \al_2^2
\end{pmatrix}
\qquad
\begin{pmatrix}
j^x_{\sigma_1} \\
j^x_{\sigma_2}
\end{pmatrix} =
\begin{pmatrix}
  \al_2\\0
\end{pmatrix}.
\end{equation*}
Doing the same thing for $x=r_1r_0r_2$ we have
\begin{equation*}
(-1)^{\ell(x)}
\begin{pmatrix}
\xi^x(t_1) \\ \xi^x(t_2)
\end{pmatrix}=
\begin{pmatrix}
0 \\ -\al_1\al_2^2
\end{pmatrix}
\qquad
\begin{pmatrix}
j^x_{\sigma_1} \\
j^x_{\sigma_2}
\end{pmatrix} =
\begin{pmatrix}
  0 \\ \al_2
\end{pmatrix}.
\end{equation*}
\end{ex}
\begin{ex} $Sp_{2n}$ for $n=2$ has affine Cartan matrix
\begin{align*}
  \begin{pmatrix}
   2 & -1 & 0 \\
   -2 & 2 & -2 \\
   0 & -1 & 2
  \end{pmatrix}.
\end{align*}
We have $\delta=\al_0+\theta=\al_0+2\al_1+\al_2$.
\begin{equation*}
\begin{array}{|c||c|c|c|c|} \hline
p & \sh_p & \sigma_p & t_{p-1} & \sh_p r_\theta \\ \hline
1 & \id & r_0 & r_0r_1r_2r_1 & r_1r_2r_1 \\
2 & r_1 & r_1r_0 & r_1r_0r_1r_2 & r_2r_1 \\
3 & r_2r_1 & r_2r_1r_0 & r_2r_1r_0r_1 & r_1 \\ \hline
\end{array}
\end{equation*}
We have
\begin{align*}
 M = \begin{pmatrix}
 2\al_1+\al_2 & 0 & 0 \\
 \al_2 & -\al_1\al_2 &0\\
 -\al_2 & \al_2(\al_1+\al_2)& -\al_2^2(\al_1+\al_2)
 \end{pmatrix}
\end{align*}
\begin{align*}
 N = \begin{pmatrix}
 \al_1(\al_1+\al_2)(2\al_1+\al_2) & 0 & 0 \\
 (\al_1+\al_2)(2\al_1+\al_2) & \al_2(\al_1+\al_2) & 0 \\
 2\al_1+\al_2&\al_1+\al_2&\al_1
 \end{pmatrix}
\end{align*}
\begin{align*}
 D = \begin{pmatrix}
 \al_1(\al_1+\al_2)(2\al_1+\al_2)^2 & 0 & 0 \\
 0 & -\al_1\al_2^2(\al_1+\al_2) & 0 \\
 0 & 0 & -\al_1\al_2^2(\al_1+\al_2)
 \end{pmatrix}
\end{align*}
\begin{align*}
ND^{-1} = \begin{pmatrix} (2\al_1+\al_2)^{-1} & 0 & 0 \\
(\al_1(2\al_1+\al_2))^{-1} & -(\al_1\al_2)^{-1} & 0 \\
(\al_1(\al_1+\al_2)(2\al_1+\al_2))^{-1} & -(\al_1\al_2^2)^{-1} &
-(\al_2^2(\al_1+\al_2))^{-1}
\end{pmatrix}
\end{align*}
Now let $x=r_0r_1r_2$. We have
\begin{equation*}
(-1)^{\ell(x)}
\begin{pmatrix}
\xi^x(t_1) \\
\xi^x(t_2) \\
\xi^x(t_3)
\end{pmatrix}
=
\begin{pmatrix}
  (\al_1+\al_2)(2\al_1+\al_2)^2 \\ \al_2^2(\al_1+\al_2) \\ 0
  \end{pmatrix}
\end{equation*}
The matrix $ND^{-1}$ acting on the above column vector, gives
the vector
\begin{equation*}
\begin{pmatrix}
j_{\sigma_1}^x \\
j_{\sigma_2}^x \\
j_{\sigma_3}^x
\end{pmatrix} =
\begin{pmatrix}
(\al_1+\al_2)(2\al_1+\al_2) \\
2(\al_1+\al_2) \\
1
\end{pmatrix}
\end{equation*}
Now let $x=r_1r_2r_1$. We have
\begin{equation*}
(-1)^{\ell(x)}
\begin{pmatrix}
\xi^x(t_1) \\
\xi^x(t_2) \\
\xi^x(t_3)
\end{pmatrix}
=
\begin{pmatrix}
\al_1(\al_1+\al_2)(2\al_1+\al_2) \\ 0 \\ 0
\end{pmatrix}
\end{equation*}
\begin{equation*}
\begin{pmatrix}
j_{\sigma_1}^x \\
j_{\sigma_2}^x \\
j_{\sigma_3}^x
\end{pmatrix} =
\begin{pmatrix}
\al_1(\al_1+\al_2) \\
(\al_1+\al_2) \\
1
\end{pmatrix}
\end{equation*}
\end{ex}

\begin{ex} $SO_{2n+1}$ for $n=3$ has affine Cartan matrix
\begin{align*}
  \begin{pmatrix}
  2 & 0 & -1 & 0 \\
   0 & 2 & -1 & 0 \\
   -1 & -1 & 2 & -1 \\
   0 & 0 & -2 & 2
  \end{pmatrix}.
\end{align*}
We have $\delta=\al_0+\theta=\al_0+\al_1+2\al_2+2\al_3$.
\begin{equation*}
\begin{array}{|c||c|c|c|c|} \hline
p & \sh_p & \sigma_p & t_{p-1} & \sh_p r_\theta \\ \hline
1 & \id & r_0 & r_0r_2r_3r_2r_1r_2r_3r_2 & r_2r_3r_2r_1r_2r_3r_2 \\
2 & r_2 & r_2r_0 & r_2r_0r_2r_3r_2r_1r_2r_3 & r_3r_2r_1r_2r_3r_2 \\
3 & r_3r_2 & r_3r_2r_0 & r_3r_2r_0r_2r_3r_2r_1r_2 & r_2r_1r_2r_3r_2 \\
4 & r_2r_3r_2&r_2r_3r_2r_0&r_2r_3r_2r_0r_2r_3r_2r_1 & r_1r_2r_3r_2\\
5 & r_0r_2r_3r_2&r_0r_2r_3r_2r_0&r_0r_2r_3r_2r_0r_1r_2r_3r_2r_1&r_2r_3r_2 \\ \hline
\end{array}
\end{equation*}
To save space let us write $\al_{ijk}:=i\al_1+j\al_2+k\al_3$. We have
\begin{align*}
M=
\begin{pmatrix}
\al_{122}&&&\\
\al_{112}&-\al_{010}\al_{112}&&\\
\al_{110}&-\al_{110}\al_{012}&\al_{110}\al_{012}\al_{001}&\\
\al_{100}&-2\al_{100}\al_{011}&\al_{100}\al_{011}\al_{012}&-\al_{100}\al_{010}\al_{011}\al_{012}
\end{pmatrix}
\end{align*}
\begin{align*}
N=
\begin{pmatrix}
\al_{110}\al_{111}\al_{112}\al_{122}\al_{010}\al_{011}\al_{012} &&& \\
 \al_{110}\al_{111}\al_{112}\al_{122}\al_{011}\al_{012}
&\al_{100}\al_{111}\al_{112}\al_{122}\al_{012}\al_{001}&&\\
2\al_{110}\al_{111}\al_{112}\al_{122}\al_{011}
&\al_{100}\al_{111}\al_{112}\al_{122}\al_{012}
&\al_{100}\al_{110}\al_{111}\al_{122}\al_{010}&\\
 \al_{110}\al_{111}\al_{112}\al_{122}
&\al_{100}\al_{111}\al_{112}\al_{122}
&\al_{100}\al_{110}\al_{111}\al_{122}
&\al_{100}\al_{110}\al_{111}\al_{112}
\end{pmatrix}
\end{align*}
$D$ has diagonal entries
\begin{align*}
\al_{110}\al_{111}\al_{112}\al_{122}^2\al_{010}\al_{011}\al_{012} \\
-\al_{100}\al_{111}\al_{112}^2\al_{122}\al_{010}\al_{012}\al_{001}\\
\al_{100}\al_{110}^2\al_{111}\al_{122}\al_{010}\al_{012}\al_{001}\\
-\al_{100}^2\al_{110}\al_{111}\al_{112}\al_{010}\al_{011}\al_{012}
\end{align*}
One may verify that $MN=D$.
\end{ex}

\end{document}